\documentclass[11pt,a4paper]{article}
\usepackage[T1]{fontenc}
\usepackage[margin=28mm]{geometry}
\usepackage{amsmath,amssymb,amsthm,mathtools}
\usepackage{microtype,booktabs,enumitem}
\usepackage[hidelinks]{hyperref}
\usepackage{fancyhdr}

\numberwithin{equation}{section}
\newtheorem{theorem}{Theorem}[section]
\newtheorem{conjecture}{Conjecture}[section]
\newtheorem{proposition}[theorem]{Proposition}
\newtheorem{lemma}[theorem]{Lemma}
\theoremstyle{remark}\newtheorem{remark}[theorem]{Remark}
\newcommand{\Ham}{\operatorname{Ham}}
\newcommand{\Fix}{\operatorname{Fix}}
\newcommand{\Crit}{\operatorname{Crit}}
\newcommand{\supp}{\operatorname{supp}}
\newcommand{\id}{\operatorname{id}}
\newcommand{\R}{\mathbb R}
\newcommand{\dd}{\mathrm d}
\DeclareMathOperator{\cl}{cl}
\DeclareMathOperator{\diag}{diag}

\newcommand{\CP}{\mathbb{CP}}

\title{A Smooth Hamiltonian Diffeomorphism with Two \\Fixed Points on $S^2\times S^2$}

\author{Hao Jiao\\
Qiuzhen College, Tsinghua University\\
\texttt{jiaoh21@mails.tsinghua.edu.cn}}
\date{}

\begin{document}

\maketitle
\begin{abstract}
We give a construction of a Hamiltonian diffeomorphism of the equal-area symplectic product $S^2\times S^2$ with two fixed points. The argument starts with a four-fixed-point map, passes to rescaled cotangent coordinates near the antidiagonal, and reduces the local fixed-point equations to two planar saddle problems. A relative modification merges each pair of saddles, and mixed generating functions are used to transfer the modification to the original map. We include a method to merge saddle points along a level arc, uniform existence of mixed generating functions, the order of parameter choices, and the global graph argument. We also give a simpler explicit construction for the unequal-area
product. These constructions give counterexamples to the degenerate version of Arnold's conjecture proposed by Arnold~\cite{Arnold86} when the manifold is $S^2\times S^2$.

This work is done with ChatGPT 6 Astra. The author gave the idea to study the limit behavior of $\Phi_\rho$. GPT suggested a key idea of using a generating function of the second kind to  construct this Hamiltonian. Also, most of the computations were carried out with assistance from GPT.
\end{abstract}
\clearpage
\tableofcontents
\clearpage

\section{Introduction}

For a compact symplectic manifold $(M,\omega)$, $H :M\times S^1\rightarrow \R$ is a 1-periodic smooth Hamiltonian. Its Hamiltonian flow is defined by 
\begin{center}
    $\frac{d}{dt}\varphi_H^t = X_H^t(\varphi_H^t),\qquad \varphi_H^0 = id_M$,
\end{center}
where $dH_t = \iota_{X_H^t}\omega$. The number of fixed points of $\varphi_H^1$ is one of the central questions of symplectic topology. A fixed point $x$ of $\varphi_H^1$ is called nondegenerate if $D\varphi_H^1(x)-I:T_xM\rightarrow T_xM$ has trivial kernel. And $H$ is called nondegenerate if all fixed points of $\varphi_H^1$ are nondegenerate. In that case, the homological Arnold
inequality is
\begin{equation}\label{eq:arnold-betti}
 \#\Fix(\varphi_H^1)\geq\sum_k\dim_{\mathbb F}H^k(M;\mathbb F).\nonumber
\end{equation}

The early finite-dimensional and variational approaches established
fundamental special cases. Eliashberg proved the surface case, and
Floer subsequently gave a proof for surfaces and certain K\"ahler
manifolds~\cite{Eliashberg79,Floer86}. Conley and
Zehnder~\cite{ConleyZehnder1983} obtained the cup-length and nondegenerate
Betti-number bounds for standard symplectic tori.
Fortune and Weinstein~\cite{Fortune85}~\cite{FortuneWeinstein85} established the sharp $n+1$ lower bound
for $\CP^n$ with its standard symplectic form, allowing degenerate
fixed points.

Floer's introduction of the homology of the symplectic action
functional transformed the subject. His work treated the aspherical
and monotone settings and incorporated the effect of holomorphic
spheres~\cite{Floer88,Floer1989}. The comparison of Floer homology
with quantum cohomology was made explicit by
Piunikhin--Salamon--Schwarz~\cite{PSS96}.
Hofer--Salamon~\cite{HoferSalamon1995} and
Ono~\cite{Ono1995} extended the theory in the weakly monotone setting.
For general closed symplectic manifolds, the virtual methods of
Fukaya--Ono~\cite{FukayaOno}, Liu--Tian~\cite{LT98}, and
Ruan~\cite{Ruan99} established the nondegenerate homological inequality
over $\mathbb{Q}$. Pardon's implicit-atlas construction~\cite{Pardon16}
and the polyfold proof of Filippenko--Wehrheim~\cite{FW22} later
provided further foundations and proofs of this rational-coefficient
statement.

The coefficient question has also seen substantial progress.
Abouzaid and Blumberg~\cite{AbouBlum21} proved the nondegenerate
Betti-number bound over fields of arbitrary positive characteristic.
Bai--Xu~\cite{Xu-Bai} and
Rezchikov~\cite{Rezchikov} subsequently developed integral versions of
Hamiltonian Floer theory, yielding bounds that also detect integral
torsion. The subsequent foundational work of
Bai--Xu~\cite{BX26} also constructs integral continuation maps,
the PSS isomorphism, and the pair-of-pants product. So far, this inequality is proved in all coefficients.

Arnold proposed a degenerate analogue of the previous inequality (see, in particular,~\cite{Arnold86}).
To state the version relevant here without a nondegeneracy assumption,
write
\[
 \operatorname{crit}(M)=
 \min_{f\in C^\infty(M,\R)}\#\Crit(f).
\]

\begin{conjecture}
 Let $(M,\omega)$ be a closed symplectic manifold. For every smooth Hamiltonian $H$ (not necessarily nondegenerate), the number of fixed points of $\varphi_H^1$ is greater than or equal to the number of critical points of some smooth function on M, namely
    \begin{center}
        $\operatorname{crit}(M)\leq$ \#\{fixed points of $\varphi_H^1$\}.
   \end{center}    
\end{conjecture}
If $\cl_{\mathbb F}(M)$ denotes the largest number of positive-degree
cohomology classes over a field $\mathbb F$ whose cup product is nonzero,
Lusternik--Schnirelmann theory~\cite{RO99} gives
$\operatorname{crit}(M)\geq \cl_{\mathbb F}(M)+1$.
Consequently the cup-length inequality
\begin{equation}\label{eq:arnold-cup}
 \#\Fix(\varphi_H^1)\geq \cl_{\mathbb F}(M)+1\nonumber
\end{equation}
is a weaker prediction of Conjecture 1.1. 

For degenerate Hamiltonians, the established results have a different
scope. The $\mathbb{CP}^n$ case was independently established by Fortune and Floer~\cite{Fortune85,Floer1989}. The Lagrangian-intersection methods of Floer and
Hofer~\cite{FloerCup89,Hofer88} give cup-length estimates in the
absence of the relevant bubbling, including the fixed-point problem
when $\pi_2(M)=0$. Rudyak--Oprea~\cite{RO99}, building on these
methods and category-theoretic refinements, obtained the original
critical-point lower bound in the symplectically aspherical setting.
L\^e--Ono~\cite{LeOno} proved further cup-length estimates under
monotonicity and minimal-Chern-number hypotheses, including negatively
monotone manifolds of dimension $2n$ with minimal Chern number at least
$n$. Schwarz~\cite{Schwarz98} established the ordinary cup-length
estimate under a sufficiently small Hofer-energy assumption on
rational symplectic manifolds, and introduced a quantum cup-length
bound for weakly monotone manifolds without that energy restriction.

Givental~\cite{Givental} obtained lower bounds for toric manifolds
with integral symplectic class using finite-dimensional
approximations and equivariant methods. More recently,
Gong~\cite{Gong25} related degenerate fixed-point bounds to mixed
Gromov--Witten invariants and deformed Hamiltonian spectral
invariants, extending Givental's estimates for Fano toric manifolds.
Related estimates using the quantum-cohomology action on Floer
cohomology and energy bounds were studied in~\cite{HaoJiao26}.

This paper studies the degenerate version of Arnold's conjecture in the case where $M$ is $S^2\times S^2$. Namely, let $(S^2,\omega_{S^2})$ be the unit sphere of total area $4\pi$, and put
\begin{center}
$(M,\omega)=\bigl(S^2\times S^2,\pi_1^*\omega_{S^2}+\pi_2^*\omega_{S^2}\bigr).$
\end{center}
We give a construction of Hamiltonian on this manifold with only two fixed points, that is 

\begin{theorem}\label{thm:claim}
There exists a Hamiltonian diffeomorphism $\Phi$ of $(M,\omega)$ with exactly two fixed points. It is the time-one map of a smooth Hamiltonian $H\in C^\infty((\R/\mathbb Z)\times M)$.
\end{theorem}

In our equal-area $S^2\times S^2$ case, the above results gave the lower bound 2 rather than the cup length estimate 3.
Theorem 1.1 shows that the lower bound two is attained by
a smooth Hamiltonian diffeomorphism of the equal-area product. Thus the theorem determines
the minimum number of fixed points in $\Ham(M,\omega)$.
For comparison, Buhovsky--Humili\`ere--Seyfaddini~\cite{BHS18}
constructed Hamiltonian homeomorphisms with one fixed point on every
closed connected symplectic manifold of dimension at least four;
their result concerns the $C^0$ category, whereas the construction
here stays in the smooth Hamiltonian category. Moreover, the Lagrangian version of the question was considered independently by Shelukhin and Abouzaid.

The construction uses a generating function to describe the fixed points of the time-one map. By doing some operations on the generating function, we can work backward to obtain the Hamiltonian we need. The two fixed points in the construction are isolated and degenerate.

The proof is a relative modification of the Lagrangian graph of an
explicit Hamiltonian diffeomorphism. It is mainly divided into four steps. 

$\bullet$ The first step is constructing the initial Hamiltonian diffeomorphism $\Phi_\rho$ with 4 fixed points 
\[
 \Phi_\rho=\tau_\rho\circ F,\qquad
 F=\sigma\circ(R_1\times R_2),\qquad
 \sigma(j,k)=(k,j).
\]
Here $R_1,R_2$ are rotations through $\pi$ about the axes
$b_1=(\sqrt3/2,-1/2,0)$ and $b_2=(\sqrt3/2,1/2,0)$.

$\tau_\rho$ is a Dehn twist around the anti-diagonal. It acts nontrivially on $H^2(M)$ and is therefore not Hamiltonian. To solve this, the map $\sigma$ is necessary here. We remark here that the choice of $b_1,b_2$ is not essential. These axes are chosen to simplify the subsequent computations.

Although $\tau_\rho$ and $\sigma$ are not individually Hamiltonian,
$\tau_\rho\sigma$ is generated by the smooth radial Hamiltonian
\[
 K_\rho(r)=\pi\int_0^r\bigl(\chi_\rho(s)-1\bigr)\,\dd s.
\]
Consequently $\Phi_\rho$ is Hamiltonian. Its fixed-point equations
reduce to two monotone scalar equations and give exactly four
nondegenerate fixed points, two approaching each of
$p_+=(e_z,-e_z)$ and $p_-=(-e_z,e_z)$ as $\rho\to0$. In this step, these points have Conley-Zehnder index of 2 around $p_+$ and 0 around $p_-$  modulo 4.

$\bullet$ The second step is to study the limit behavior of $\Phi_\rho$ as $\rho\rightarrow 0$.
The fixed points of $\Phi_\rho$ are about $\rho$ away from the anti-diagonal. After a rescaling $(q,p)=\rho(u,v)$, they lie in fixed bounded regions. In these regions, the maps under $(u,v)$ coordinates converge smoothly to a linear symplectic map followed by the shear
$(u,v)\mapsto(u+\nabla h(v),v)$, where
\[
 h(V)=-\pi\int_{|V|}^{2}\kappa(s)\,\dd s.
\]

$\bullet$ The third step is to use a generating function to turn the fixed points problem into a critical points problem.
Use the input position $u$ and output momentum $V$ as independent
variables, and write the mixed generating function as $S_0(u,V)$.
Its relations are $v=S_{0,u}$ and $U=S_{0,V}$. Thus fixed points of $\Phi_0$ correspond to critical points of $W_0=S_0-u\cdot V$.
The $u$-Hessian is definite,
$W_{0,uu}=\mp2\sqrt3\,I$, so $W_{0,u}=0$ can be solved uniquely
for $u$. The reduced functions $g(V) = W_0(u(V),V)$ are
\begin{equation}\label{eq:reduced-models}
 \begin{aligned}
 g_+(V)&=h(V)+\frac{\sqrt3}{2}V_1^2
                    -\frac{1}{2\sqrt3}V_2^2,\nonumber\\
 g_-(V)&=h(V)-\frac{\sqrt3}{2}V_1^2
                    +\frac{1}{2\sqrt3}V_2^2.\nonumber
 \end{aligned}
\end{equation}
Each has two saddles of equal critical value joined by an embedded
level arc in $1<|V|<2$. 
Then we can use the idea of Monkey saddle~\cite{Takens68} to merge the two saddle points along the level arc between them.
Denote the resulting functions by $\widetilde g_\pm$.

$\bullet$ The final step is to work backward from $\Phi_0$ to some $\Phi_\rho$ with small $\rho >0$.
A modification of the limiting model alone is insufficient. We therefore construct
mixed generating functions $S_\rho$ for the actual maps, uniformly
on a fixed product domain in $(u,V)$. And we eliminate $u$ to obtain
$g_\rho\to g_0$. The domain is chosen using the
\emph{output} momentum $V$. It is not necessary for the input
momentum $v$ to lie in the same annulus.
With a cutoff $\beta_0$ equal to one on the modification disk, set
\[
 \delta g_\rho=(\widetilde g_\pm-g_0)+\beta_0(g_0-g_\rho).
\]
Then $g_\rho+\delta g_\rho=\widetilde g_\pm$ exactly on that disk.
On the outer collar the gradient of $g_0$ is bounded away from zero,
so the correction introduces no further critical points.

To realize this planar operation on the original graph, use
\begin{equation}\label{eq:graph-deformation}
 S_{\rho,\lambda}(u,V)
   =S_\rho(u,V)+\lambda\alpha_L(u)\delta g_\rho(V),
 \qquad 0\leq\lambda\leq1,\nonumber
\end{equation}
where $\alpha_L=1$ for $|u|\leq L$ and $\alpha_L=0$ for
$|u|\geq2L$. First fix the planar modifications, then take $L$
large, and finally take $\rho$ small.
Since $S_{0,uV}=\diag(-2,2)$ and
$\|\dd\alpha_L\|=O(L^{-1})$, the mixed Hessian stays invertible
throughout the deformation. The definite $u$-quadratic part makes
$|W_{\rho,\lambda,u}|$ grow linearly with $|u|$, excluding critical
points in the cutoff region. In the core, the $u$-critical equation
is unchanged and the reduced function at $\lambda=1$ is exactly
$g_\rho+\delta g_\rho$. Hence precisely one fixed point remains
in each modified chart.

Finally, the modifications agree with the original graph near their
boundaries and define a global Lagrangian immersion
$\iota_\lambda=(p_\lambda,q_\lambda):M\to M\times M$.
The invertible mixed Hessian makes $p_\lambda$ a local
diffeomorphism. Compactness makes it a covering, and the homotopy
to $p_0=\mathrm{id}$ gives degree one. It is therefore a global
diffeomorphism, so the modified Lagrangian is the graph of
$q_\lambda p_\lambda^{-1}$, a global symplectic diffeomorphism.
The resulting symplectic isotopy is Hamiltonian because
$H^1(S^2\times S^2;\R)=0$. A smooth reparametrization in time,
constant near the endpoints, produces the required periodic
Hamiltonian. This completes the passage from the two planar
mergers to a smooth global map with exactly two fixed points.

Since this paper was written with AI assistance, I want to say something about how I use AI to study this problem and how AI helps me. At the beginning, my initial idea is to start with the canonical map with four fixed points, namely $R_1\times R_2$, and to study the behavior as the fixed points get closer and  closer.  AI suggested that we can use a Dehn twist around the anti-diagonal to describe this process. That's indeed the case because we can write down the exact formulation of $\Phi_\rho$, and it's very convenient for us to perform further operations.

Then I studied the limit behavior of these $\Phi_\rho$'s and asked AI which method is suitable to deal with this kind of issue. It suggested the method of generating functions and showed me how to reduce the fixed points problem to a critical points problem.

After that, the remaining steps are easier because we have a standard technique to merge two saddle points of a planar function~\cite{Takens68}. And there is not much difficulty in extending it to $g_\rho$ and working backward to $\Phi_\rho$.

As one can see, we do not use any deep theory in this construction. I think the biggest help AI has given me is that it can gather and use tools I'm not familiar with. That is also what I think AI is better at than humans. 

The first version of this paper was made public on arXiv  on Sep 27.  After that, on Oct. 7, OpenAI made some  documents public. Among them, there are documents that claim  counterexamples to the degenerate case of Arnold's conjecture  in dimension not less than 6.
We remark that our example (which is 4-dimensional) easily  implies that a similar example exists in any even dimension not less than 4. Our results are independent of theirs, and we took a very different strategy to construct the counterexample.

This problem was suggested by my advisor when I was an undergraduate student. After about two years of work, I was very fortunate to find this example. The author thanks  Wenmin Gong, Honghao Gao, Hang Yuan, and Shaoyun Bai for helpful discussions during the course of this research, and also deeply thanks  his advisor Kenji Fukaya for his guidance and valuable suggestions.

\section{The four-fixed-point starting map}

In this section, we first give the expression of a family of  Hamiltonian diffeomorphisms on $M$ with four fixed points.  Write a point of $M$ as $(j,k)\in\R^3\times\R^3$, $|j|=|k|=1$, let $e_z$ be the unit vector in $z$ direction. Set
\begin{center}
$\sigma(j,k)=(k,j),\qquad b_1=(\sqrt3/2,-1/2,0),\quad b_2=(\sqrt3/2,1/2,0),$\\$
p_+=(e_z,-e_z),\qquad p_-=(-e_z,e_z).$
\end{center}
Let $R_i$ be rotation by $\pi$ around $b_i$ and $F = \sigma\circ(R_1\times R_2)$. Fix a smooth nonincreasing function $\kappa:[0,\infty)\to[0,100]$ such that
\begin{center}
    $\kappa(R)=100\quad(R\leq5/4),\qquad \kappa(R)=0\quad(R\geq7/4),$
\end{center}
strictly decreasing between these radii.
For $0<\rho<1/100$, choose a smooth nonincreasing $\chi_\rho$ with
\begin{center}
$\chi_\rho(r)=1\quad(0\leq r\leq\rho/2),\qquad
\chi_\rho(r)=\rho\kappa(r/\rho)\quad(r\geq\rho).$
\end{center}
Use an endpoint-flat monotone interpolation on $[\rho/2,\rho]$.

Put $S=j+k$ and $r=|S|$. For $r>0$, let $\tau_\rho$ rotate $j$ and $k$ simultaneously around $S/r$ through angle $\pi\chi_\rho(r)$. Extend it by $\sigma$ at $r=0$. For $r\leq\rho/2$ the rotation interchanges $j$ and $k$, so this extension is smooth.
\begin{proposition}
The map $\Phi_\rho=\tau_\rho\circ F$ is Hamiltonian.
\end{proposition}
\begin{proof}
Equip $M=S^2\times S^2$ with the product symplectic form $\omega=\omega_{S^2}\oplus\omega_{S^2}$, where $(\omega_{S^2})_j(\xi,\eta)=j\cdot(\xi\times\eta)$. We use the convention $\iota_{X_H}\omega=dH$. For a unit vector $a$, let $R_{a,\theta}$ denote the right-handed rotation through angle $\theta$ about $a$. Set $S=j+k$ and $r=|S|$, and define
\[K_\rho(j,k)=K_\rho(r)=\pi\int_0^r\bigl(\chi_\rho(s)-1\bigr)\,ds.\]
Since $\chi_\rho=1$ on $[0,\rho/2]$, the function $K_\rho$ vanishes on a neighborhood of the anti-diagonal $\Delta_-=\{(j,-j):j\in S^2\}$. Away from $\Delta_-$, the function $r$ is smooth. Thus $K_\rho$ is smooth on all of $M$, despite the nonsmoothness of $r$ along $\Delta_-$.

Suppose $r>0$ and write $a=S/r$. We have
\[X_{K_\rho}(j,k)=\pi\bigl(\chi_\rho(r)-1\bigr)\bigl(a\times j,a\times k\bigr).\]
Along a Hamiltonian trajectory, we have
\[\dot S=\dot j+\dot k=\pi\bigl(\chi_\rho(r)-1\bigr)\frac{S}{r}\times S=0.\]
Thus $S$, $r$, and $a$ remain constant. Writing $c=\pi(\chi_\rho(r)-1)$, the Hamiltonian equations become $\dot j=c\,a\times j$ and $\dot k=c\,a\times k$. Their solution is
\[\varphi_{K_\rho}^{\,t}(j,k)=\bigl(R_{a,tc}j,R_{a,tc}k\bigr).\]
In particular, the time-one map rotates both factors about $S/r$ through angle $\pi(\chi_\rho(r)-1)$.

We now compare this map with $\tau_\rho\circ\sigma$, where $\sigma(j,k)=(k,j)$. Since $|j|=|k|=1$, we have $R_{a,\pi}j=k$,
and $R_{a,\pi}k=j$. Thus, whenever $r>0$, the swap $\sigma$ is simultaneous rotation of both factors through angle $\pi$ about $a$.

By definition, $\tau_\rho$ rotates both factors about the same axis through angle $\pi\chi_\rho(r)$. Since $\sigma$ preserves $S$ and $r$, their composition is
\[(\tau_\rho\circ\sigma)(j,k)=\bigl(R_{a,\pi(\chi_\rho(r)+1)}j,R_{a,\pi(\chi_\rho(r)+1)}k\bigr).\]
So $\tau_\rho\circ\sigma=\varphi_{K_\rho}^{\,1}$ on $M\setminus\Delta_-$.

Near $\Delta_-$, we have $\chi_\rho=1$ and hence $\tau_\rho=\sigma$, including the prescribed extension at $r=0$. Therefore $\tau_\rho\circ\sigma=\operatorname{id}$ there. Since $K_\rho$ vanishes on the same neighborhood, its Hamiltonian flow is also the identity there. The equality consequently holds on all of $M$, and we conclude that
\[\tau_\rho\circ\sigma=\varphi_{K_\rho}^{\,1}\in\operatorname{Ham}(M,\omega).\]

And of course the rotation $R_1\times R_2$ is a Hamiltonian map. Thus
\begin{center}
$\Phi_\rho=(\tau_\rho\sigma)\circ(R_1\times R_2)
=\varphi_{K_\rho}^1\circ(R_1\times R_2)\in\Ham(M,\omega).$
\end{center}
\end{proof}

\begin{proposition}
For sufficiently small $\rho>0$, $\Phi_\rho$ has exactly four fixed points, all nondegenerate, with two converging to $p_+$ and two to $p_-$ as $\rho\rightarrow 0$.
\end{proposition}
\begin{proof}
Define
\[P:=S\circ\Phi_\rho-S=S\circ F-S=\left(-\frac{x+\sqrt{3}d_y}{2},-\frac{3y+\sqrt{3}d_x}{2},-2z\right).\] 
Since $\tau_\rho$ preserves $S$, a fixed point of $\Phi_\rho$ must satisfy $S(F(j,k)) =:S_F(j,k)=S(j,k)$. Write $S=(x,y,z)$ and $d=j-k = (d_x,d_y,d_z)$. Substitution of the rotation matrices gives
\begin{center}
$z=0,\qquad d_x=-\sqrt3\,y,\qquad d_y=-x/\sqrt3.$
\end{center}
The unit-sphere constraints are
\begin{center}
$S\cdot d=0,\qquad |S|^2+|d|^2=4.$
\end{center}
They imply $xy=0$ and determine $d_z$ by $x,y$ up to sign. 

If $S=0$ the candidates are $p_\pm$, which $\tau_\rho F$ interchanges, so neither is fixed.

For $S=(0,\pm r,0)$, the admissible solutions are
\begin{center}
$j=\left(\mp\frac{\sqrt3r}{2},\ \pm\frac r2,\ -\sqrt{1-r^2}\right),\qquad k=R_{e_y}(\pi)j,$
\end{center}
where
\begin{center}\label{eq:yroot}
$\pi\chi_\rho(r)=2\arctan\frac{\sqrt3r}{2\sqrt{1-r^2}}.$
\end{center}

For $S=(\pm r,0,0)$, they are
\begin{center}
$j=\left(\pm\frac r2,\ \mp\frac{r}{2\sqrt3},\ \sqrt{1-r^2/3}\right),\qquad k=R_{e_x}(\pi)j,$
\end{center}
where
\begin{center}\label{eq:xroot}
$\pi\chi_\rho(r)=2\arctan\frac{r}{2\sqrt3\sqrt{1-r^2/3}}.$
\end{center}
The other signs of $d_z$ require angles in $(\pi,2\pi)$ and do not give solutions because the rotation angles of $\tau_\rho$ lie in $[0,\pi]$. In each scalar equation the left side is nonincreasing and the right side is strictly increasing. At $r=5\rho/4$ the left side is larger, and at $r=7\rho/4$ it vanishes. Each equation therefore has one root, with
\begin{center}
$5/4<r/\rho<7/4.$
\end{center}
The two sign choices on each axis give four points. Then we prove that each of the four fixed points of $\Phi_\rho=\tau_\rho\circ F$ is nondegenerate.

Thus every fixed point lies in $P^{-1}(0)$. At each of the four points, exactly one of $x,y$ is nonzero and $d_z\ne0$. Differentiating $P=0$ together with the constraints $S\cdot d=0$ and $|S|^2+|d|^2=4$ shows that $DP$ has rank three on $T(S^2\times S^2)$. Consequently, $P^{-1}(0)$ is locally a smooth curve $\gamma(r)$ parametrized by $r=|S|$, with
\[\ker DP_{\gamma(r)}=\mathbb{R}\gamma'(r).\]

Along this curve, the oriented axis $\mathbf a=S/r$ is constant. Let $\alpha(r)$ denote the right-hand side of the previous fixed-point equation, and set
\[f(r):=\pi\chi_\rho(r)-\alpha(r).\]
The explicit formulas for $\alpha$ give $\alpha'(r)>0$, whereas $\chi_\rho'(r)\le0$. Hence $f'(r)<0$, so its zero $r_*$ is transverse. Writing $\mathcal R_{\mathbf a,\theta}$ for simultaneous rotation of both sphere factors, the rotation calculation gives
\[\Phi_\rho(\gamma(r))=\mathcal R_{\mathbf a,f(r)}\gamma(r).\]
At the fixed point $p=\gamma(r_*)=(j,k)$, differentiation therefore yields
\[(D\Phi_\rho(p)-I)\gamma'(r_*)=f'(r_*)\,(\mathbf a\times j,\mathbf a\times k)\ne0.\]
Finally, if $(D\Phi_\rho(p)-I)\xi=0$, then
\[DP_p\xi=dS_p(D\Phi_\rho(p)-I)\xi=0.\]
Thus $\xi$ is a multiple of $\gamma'(r_*)$, and the preceding calculation forces that multiple to vanish. Therefore $\ker(D\Phi_\rho(p)-I)=0$, proving nondegeneracy.
\end{proof}
\begin{remark}
    The nondegeneracy asserted here concerns only the starting map. The fixed points of the final map we obtain are degenerate.
\end{remark}

\section{Cotangent coordinates and the limit behavior}

In this section, for any two vectors $a,b\in \mathbb R^3$, we use $a\times b$ to represent the cross product in $\mathbb R^3$ and $a\cdot b$ to represent the inner product.
Near $\{(j,-j):j\in S^2\}$, set
\begin{center}
$n=\frac{j-k}{|j-k|} = (n_x,n_y,n_z),\qquad \eta=(j+k)\times n = (\eta_x,\eta_y,\eta_z),\qquad S=n\times\eta = (S_x,S_y,S_z).$
\end{center}
$n$ parametrizes the anti-diagonal sphere in $M$ and $\eta$ represents the cotangent direction, $S = j+k$ as before.
We can identify these with cotangent coordinates on a disk bundle in $T^*S^2$. On $n_z\ne0$, take $q=(q_1,q_2):=(n_x,n_y)$ and its cotangent coordinate $p=(p_1,p_2)$. More precisely, $p$ is defined by $\eta  dn = p_1dq_1+p_2dq_2$. Then
\begin{center}
$S_x=-n_zp_2,\qquad S_y=n_zp_1,\qquad S_z=q_1p_2-q_2p_1,$\\
$r^2:=|S|^2 = |p|^2-(q\cdot p)^2.$
\end{center}
The choices $n_z=\pm\sqrt{1-|q|^2}$ give the two charts.

Rescale $(q,p)=\rho(u,v)$, and define
\begin{center}
$h(R)=-\pi\int_R^2\kappa(s)\,\dd s.$
\end{center}
so that the Hamiltonian of $\tau_\rho$ is $H_\rho(r) = -\pi\int_r^{2\rho}\chi_\rho(s)ds = \rho^2h(r/\rho) $ on compact subannuli of $1<r/\rho<2$ (our fixed points satisfy $\frac{5}{4}<r/\rho<\frac{7}{4}$, thus lie in this annulus). The rescaled Hamiltonian is
\begin{center}
$h_\rho(u,v)=h\left(\sqrt{|v|^2-\rho^2(u\cdot v)^2}\right).$
\end{center}
On fixed bounded regions where these coordinates apply, it converges smoothly to $h(|v|)$, whose time-one map is
\begin{equation}
    (u,v)\longmapsto(u+\nabla h(v),v).
\end{equation}
Here and below $h(v)$ abbreviates $h(|v|)$, $|v|$ is constant along this flow. By allowing a larger fixed bounded coordinate region of $u$ first and shrinking $\rho$ afterwards, the finite trajectories remain in the chart.

The derivative of $F$ at $p_+$ is
\begin{center}
$\begin{pmatrix}U\\V\end{pmatrix}
=\begin{pmatrix}A&B\\C&D\end{pmatrix}\begin{pmatrix}u\\v\end{pmatrix},$
\end{center}
with

\begin{align}
   A=D&=\operatorname{diag}(-1/2,1/2), \nonumber\\
B&=\operatorname{diag}(\sqrt3/4,-\sqrt3/4), \nonumber\\
C&=\operatorname{diag}(-\sqrt3,\sqrt3). \nonumber
\end{align}

Near $p_-$ the blocks $B,C$ change sign. The limiting map for $\Phi_\rho$ is this linear map followed by the shear (3.1). With input position $u$ and output momentum $V$ as independent variables, its generating function (see~\cite[Sections 4.1--4.2]{Cannas2008}) is
\begin{center}\label{eq:S0}
$S_0(u,V)=-\tfrac12u^TD^{-1}Cu+u^TD^{-1}V+\tfrac12V^TBD^{-1}V+h(V).$
\end{center}
The defining relations are
\begin{center}
$v=S_{0,u},\qquad U=S_{0,V}.$
\end{center}
Consequently fixed points are critical points of
\begin{center}
$W_0(u,V)=S_0(u,V)-u\cdot V.$
\end{center}
Indeed $W_{0,u}=v-V$ and $W_{0,V}=U-u$. The relevant Hessians are
\begin{center}
$S_{0,uV}=D^{-1}=\operatorname{diag}(-2,2),\qquad
W_{0,uu}=\begin{cases}-2\sqrt3 I,&p_+\text{ chart},\\+2\sqrt3 I,&p_-\text{ chart}.
\end{cases}$
\end{center}

We also denote them by $W_{0,\pm}$ in the following. The equation $W_{0,u}=0$ uniquely eliminates $u$. Express $u$ as a function of $V = (V_1,V_2)$. Substitution yields
\begin{align}
g_+(V)&=h(V)+\frac{\sqrt3}{2}V_1^2-\frac{1}{2\sqrt3}V_2^2,\nonumber\\
g_-(V)&=h(V)-\frac{\sqrt3}{2}V_1^2+\frac{1}{2\sqrt3}V_2^2.\nonumber
\end{align}
and the fixed points correspond to the critical points of $g_\pm$.

Let  $R=|V|>0$, and $(h_{ij})=D_V^2h(|V|)$. In the order $(u_1,u_2,V_1,V_2)$, the full Hessians in the mixed charts are
\[
D^2W_{0,\pm}=\begin{pmatrix}
\mp2\sqrt3&0&-3&0\\
0&\mp2\sqrt3&0&1\\
-3&0&h_{11}\mp\sqrt3/2&h_{12}\\
0&1&h_{12}&h_{22}\mp\sqrt3/2
\end{pmatrix},\qquad
(h_{ij})=\frac{h'(R)}R I_2+\left(h''(R)-\frac{h'(R)}R\right)\frac{VV^T}{R^2}.
\]
In particular, the $VV$ block includes both the quadratic contribution $\mp(\sqrt3/2)I_2$ and the radial Hessian, with $h'=\pi\kappa$ and $h''=\pi\kappa'\leq0$. Because
\[ g_\pm(V):=W_{0,\pm}(u_{0,\pm}(V),V), \]
writing $\operatorname{ind}$ for the number of negative eigenvalues counted with multiplicity, Sylvester's law of inertia yields
\[\operatorname{ind}D^2W_{0,\pm}=\operatorname{ind}W_{0,\pm,uu}+\operatorname{ind}D^2g_\pm. \]
 At the two relevant critical points of $g_+$, $V=(0,\pm R_+)$ and $h'(R_+)/R_+=1/\sqrt3$; for $g_-$, $V=(\pm R_-,0)$ and $h'(R_-)/R_-=\sqrt3$. Thus
\[ D^2g_+=\operatorname{diag}\left(\frac4{\sqrt3},h''(R_+)-\frac1{\sqrt3}\right),\qquad D^2g_-=\operatorname{diag}\left(h''(R_-)-\sqrt3,\frac4{\sqrt3}\right). \]
Both reduced indices equal $1$, so the critical points of $g_\pm$ are exactly saddle points. Consequently,
\[ \operatorname{ind}D^2W_{0,+}=3,\qquad \operatorname{ind}D^2W_{0,-}=1. \]

Moreover, around $p_+$ (similarly for $p_-$) this result can be computed directly by 
\[
h_{11} = \frac{1}{\sqrt{3}},\qquad h_{12} = 0,\qquad h_{22} = h''(R_+)\leq 0, \qquad \text{at } V = (0,\pm R_+).
\]

Provided the stated $C^2$ convergence holds, these indices persist at the corresponding critical points of $W_{\rho,\pm}$ for sufficiently small $\rho>0$.

We want to remark that there is some relation between the Hessian of $W_0$ and the Conley-Zehnder index of the fixed points. For a mixed generating function $v=S_u$, $U=S_V$, with $W=S-u\cdot V$, the fixed-point determinant identity in dimension four gives
\[ (-1)^{\mu_{\mathrm{CZ}}}=\operatorname{sgn}\det(I-D\Phi)=\operatorname{sgn}\frac{\det D^2W}{\det S_{uV}}. \]
Since $S_{0,\pm,uV}=\operatorname{diag}(-2,2)$ has negative determinant, the same holds for small $\rho$, and hence
\[ \mu_{\mathrm{CZ}}\equiv\operatorname{ind}D^2W_{\rho,\pm}+1\pmod2. \]

\section{Level arcs and relative saddle merging}\label{sec:merge}

In this section, we will do some operations on $g_\pm$ in local charts around $p_\pm$ to get some new functions with exactly one critical point in each chart. To do this, we use the standard technique in~\cite{Takens68}, which is the following lemma.

\begin{lemma}
Let $g\in C^\infty(U)$, where $U\subset\mathbb R^2$ is open.
Suppose that two nondegenerate saddle points $p_1,p_2$ of $g$
are joined by a smoothly embedded arc
$\gamma\subset g^{-1}(c)$ whose interior consists of regular points.
For every sufficiently small smooth closed disk neighborhood
$D\Subset U$ of $\gamma$, there exists
$\widetilde g\in C^\infty(U)$ such that
\[
\operatorname{supp}(\widetilde g-g)
\Subset \operatorname{int}D,
\]
and $\widetilde g$ has exactly one critical point in $D$.
This critical point is isolated and degenerate.
\end{lemma}

\begin{proof}
Choose $D$ so that its only critical points of $g$ are $p_1,p_2$.
Following the arc-collapse construction in
~\cite[proof of Proposition~(2.9)]{Takens68}, choose a continuous
map $\Theta:D\to D$, equal to the identity near $\partial D$,
which collapses $\gamma$ to a point $q\in\operatorname{int}D$
and restricts to a diffeomorphism
\[
D\setminus\gamma \longrightarrow D\setminus\{q\}.
\]
Since $g$ is constant on $\gamma$, the function
$\bar g=g\circ\Theta^{-1}$ on $D\setminus\{q\}$ extends
continuously across $q$ by setting $\bar g(q)=c$.
It is smooth and has no critical points away from $q$.

Identify $\operatorname{int}D$ diffeomorphically with
$\mathbb R^2$, sending $q$ to the origin.
By~\cite[Theorem~(2.7)]{Takens68}, the resulting continuous
function admits a smooth replacement with exactly one critical
point, agreeing with it outside a compact subset of $\mathbb R^2$.
Pulling back to $\operatorname{int}D$ gives a smooth function
equal to $\bar g$, hence to $g$, near $\partial D$.
Extending it by $g$ outside $D$ yields the required
$\widetilde g$.

Finally, the gradients agree near $\partial D$, so the gradient
index of the unique critical point $\widetilde p$ is
\[
\operatorname{ind}_{\widetilde p}(\nabla\widetilde g)
=
\operatorname{ind}_{p_1}(\nabla g)
+
\operatorname{ind}_{p_2}(\nabla g)
=-2.
\]
A nondegenerate planar critical point has gradient index
$\pm1$; therefore $\widetilde p$ is degenerate.
\end{proof}

Then we need to check that the limit functions $g_\pm$ satisfy the hypotheses of this lemma. Namely, we need to find a level arc connecting two saddle points of $g_\pm$. 
In polar coordinates, both planar functions take the form
\begin{center}
$g(R,\vartheta)=h(R)+\tfrac12\lambda(\vartheta)R^2,\qquad -k\leq\lambda(\vartheta)\leq a,$
\end{center}
where $(k,a)=(1/\sqrt3,\sqrt3)$ for $g_+$ and $(\sqrt3,1/\sqrt3)$ for $g_-$. In the relevant annulus the critical points lie on the negative quadratic axis, at the unique radius $R_s$ satisfying
\begin{center}
$\pi\kappa(R_s)=kR_s,$ $\frac{5}{4}<R_s<\frac{7}{4}$.
\end{center}
The radial Hessian is $\pi\kappa'(R_s)-k<0$ (by definition of $\kappa$), and the transverse Hessian is positive. Thus both points are saddles, with common critical value
\begin{center}
$c_s=h(R_s)-\tfrac12 kR_s^2.$
\end{center}
For $1<R<R_s$, one has $\pi\kappa(R)/R>k$, so $\partial_Rg>0$. The estimates are
\begin{center}
$g(1,\vartheta)\leq-25\pi+\sqrt3/2<-77,\quad
h(R_s)\geq-kR_s(2-R_s),\quad c_s>-4.$\\
In particular, $g(1,\vartheta)<c_s$.
\end{center}
Except along the two negative-axis directions, $g(R_s,\vartheta)>c_s$. Thus each nonaxis direction has a unique inner root of $g=c_s$. These roots form two inner arcs between the saddles. Choose one and extend it slightly beyond both endpoints along the same smooth level branches. The Morse lemma gives smoothness at the endpoints. A sufficiently small tubular neighborhood of the extended arc is an embedded disk $D$ with
\begin{center}
$\overline D\subset\{V:1<|V|<2\},$
\end{center}
containing no other critical points.

Applying the lemma to the two functions $g_\pm$ gives $\widetilde g_\pm$ such that
\begin{center}
$\supp(\widetilde g_\pm-g_\pm)\Subset D,\qquad
\#\Crit(\widetilde g_\pm|_D)=1.$
\end{center}
The disks and coordinates may be chosen separately in the two charts.

\section{Generating functions for the actual map and exact correction}

The preceding saddle modification concerns the limit $\rho=0$. The purpose of this section is to construct a planar function $g_\rho$ for the actual map at $\rho>0$, whose critical points correspond to its fixed points in the selected graph chart.

Recall that $D$ is the disc on which we deform $g_\pm$ in Section 4. Choose a slightly larger disk $D_1$ with
\begin{center}
$\overline D\subset D_1,\qquad \overline{D_1}\subset\{V:1<|V|<2\}.$
\end{center}
Fix a finite constant $L$, to be made sufficiently large later, and let
\begin{center}\label{eq:Omega}
$\Omega=B_{3L}(0)\times D_1.$
\end{center}
All small-$\rho$ assertions in this section are made after fixing $L$. Write the Hamiltonian diffeomorphism $\Phi_\rho$ as
\begin{center}
$(u,v)\longmapsto\bigl(\mathcal U_\rho(u,v),P_\rho(u,v)\bigr).$
\end{center}

\begin{lemma}
    For sufficiently small $\rho>0$, the graph of $\Phi_\rho$ can be parameterized by $u,V$ in the mixed chart $\Omega$. Namely, its graph is equal to $(u,v_\rho(u,V),U_\rho(u,V),V)$ for some smooth functions $v_\rho,U_\rho$ in the mixed chart.
\end{lemma}

\begin{proof}
    To parameterize its graph by input position $u$ and output momentum $V$, one must solve $P_\rho(u,v)=V$. In the limit,
\begin{center}
$P_0(u,v)=Cu+Dv,\qquad v_0(u,V)=D^{-1}(V-Cu).$
\end{center}
As $(u,V)$ ranges over the closure of a slightly enlarged version of $\Omega$, these solutions of $v$ remain in a fixed bounded set. On a fixed neighborhood of this family, the smooth convergence of the local maps gives
\begin{center}
$P_\rho(u,v)=Cu+Dv+E_\rho(u,v),\qquad E_\rho\longrightarrow0\text{ in }C^\infty.$
\end{center}

Fix a small radius $\delta>0$. The momentum equation is equivalent to the fixed-point equation
\begin{center}
$v=T_{\rho,u,V}(v):=v_0(u,V)-D^{-1}E_\rho(u,v).$
\end{center}
For sufficiently small $\rho$, uniformly over the parameter domain and the balls $|v-v_0(u,V)|\leq\delta$,
\begin{center}
$|D^{-1}E_\rho|\leq\delta/2,\qquad
\|D^{-1}\partial_vE_\rho\|\leq1/2.$
\end{center}
Thus $T_{\rho,u,V}$ is a contraction of the closed ball into itself and has a unique fixed point $v_\rho(u,V)$. Moreover,
\begin{center}
$\partial_vP_\rho=D\bigl(I+D^{-1}\partial_vE_\rho\bigr)$
\end{center}
is invertible. The implicit function theorem makes $v_\rho$ smooth, and uniqueness glues its local descriptions on overlaps. Define
\begin{center}
$  U_\rho(u,V)=\mathcal U_\rho(u,v_\rho(u,V)).$
\end{center}
The graph chart is now parameterized by
\begin{center}
$(u,V)\longmapsto(u,v_\rho(u,V);  U_\rho(u,V),V).$
\end{center}
\end{proof}

\begin{lemma}
    For sufficiently small $\rho>0$ , there exists a smooth function $S_\rho$ of $u,\ V$ such that $ S_{\rho,u}=v_\rho,\ S_{\rho,V}=  U_\rho$.
    Thus we can define $W_\rho$ and reduced functions $g_\rho$ just like the limiting case.
\end{lemma}

\begin{proof}
    Consider the one-form
\begin{center}
$\alpha_\rho=v_\rho\cdot\dd u+  U_\rho\cdot\dd V.$
\end{center}
Symplecticity of the actual map gives, on its graph,
\begin{center}
$\sum_i\dd u_i\wedge\dd(v_\rho)_i
=\sum_i\dd(  U_\rho)_i\wedge\dd V_i.$
\end{center}
It follows that $\dd\alpha_\rho=0$. Since $\Omega$ is a ball times a disk, it is contractible, and the closed one-form has a primitive:
\begin{center}\label{eq:Srho}
$\dd S_\rho=\alpha_\rho,\qquad S_{\rho,u}=v_\rho,\qquad S_{\rho,V}=  U_\rho.$
\end{center}
This proves existence of a second-kind generating function on the whole mixed domain.

Normalize $S_\rho$ at a common base point to agree there with $S_0$. Smooth convergence of the graph parameterizations and integration of their differentials give
\begin{center}
$S_\rho\longrightarrow S_0\qquad\text{in }C^\infty$
\end{center}
on the fixed compact regions under consideration.
Set
\begin{center}
$W_\rho(u,V)=S_\rho(u,V)-u\cdot V.$
\end{center}

The matrix $W_{\rho,uu}$ remains definite, tending to $\mp2\sqrt3 I$. For each $V\in\overline{D_1}$, $\frac{\partial}{\partial u}W_\rho\rightarrow \frac{\partial}{\partial u}W_0$ as $\rho \rightarrow 0$. And the equation $\frac{\partial}{\partial u}W_0 = 0$ has solutions
\begin{center}
$u_0(V)=-H^{-1}(D^{-1}-I)V$,\\
$\text{where } H=W_{0,uu}.$
\end{center}
These solutions lie in a fixed bounded ball. The implicit function theorem, uniformly in $V$, gives nearby solutions $u_{\rho}(V)$ of $\frac{\partial W_{\rho}}{\partial u}(\cdot,V) = 0$. Choose $L$ to contain this family well inside its core. Definiteness on the convex $u$-ball gives uniqueness: suppose that \(W_{\rho,u}(u_1,V)=W_{\rho,u}(u_2,V)=0\). Set \(w=u_1-u_2\). Since the \(u\)-ball is convex, the segment \(u_2+tw\), \(0\leq t\leq1\), lies in the domain where \(W_{\rho,uu}\) is definite with a fixed sign. The fundamental theorem of calculus gives
\[
0=w^{T}\bigl(W_{\rho,u}(u_1,V)-W_{\rho,u}(u_2,V)\bigr)=\int_0^1 w^{T}W_{\rho,uu}(u_2+tw,V)w\,dt.
\]
If \(w\neq0\), the integrand is strictly positive everywhere or strictly negative everywhere, contradicting the equality. So $u_\rho = u_\rho(V)$ is a well-defined smooth function.
Then we can define the reduced function to be
\begin{center}\label{eq:grho}
$g_\rho(V)=W_\rho(u_\rho(V),V).$
\end{center}

\end{proof}
We know that
\begin{center}
$W_{\rho,u}=v_\rho-V,\qquad W_{\rho,V}=  U_\rho-u.$
\end{center}
Thus $\dd W_\rho=0$ is precisely the fixed-point equation in this chart.
By the chain rule,
\begin{center}
$\dd g_\rho=W_{\rho,u}(u_\rho(V),V)\cdot\dd u_\rho
+W_{\rho,V}(u_\rho(V),V)\cdot\dd V
=W_{\rho,V}(u_\rho(V),V)\cdot\dd V.$
\end{center}
Consequently
\begin{center}
$\Crit(g_\rho)\longleftrightarrow\Crit(W_\rho)
\longleftrightarrow\Fix(\Phi_\rho)$
\end{center}
in the selected chart.

Finally $g_\rho\to g_0$ smoothly, where $g_0$ means $g_+$ or $g_-$ in the respective chart.
Set $\delta g_0=\widetilde g_\pm-g_0$, and choose $\beta_0$ equal to $1$ on a neighborhood of $\overline D$ and compactly supported in $D_1$. Define
\begin{center}\label{eq:deltag}
$\delta g_\rho=\delta g_0+\beta_0(g_0-g_\rho).$
\end{center}
On $D$ this gives the exact identity
\begin{center}
$g_\rho+\delta g_\rho=\widetilde g_\pm.$
\end{center}
On the compact outer collar involved in the cutoff, $g_0$ has no critical points and its gradient has a positive lower bound. There the corrected function is $C^1$-close to $g_0$, so it has no new critical points. The supports are chosen so the same conclusion holds on the remaining unmodified part of $D_1$. Thus $g_\rho+\delta g_\rho$ has exactly one critical point in $D_1$.

\section{The generating-function deformation }

In this section, we just work backward from $g_\rho+\delta g_\rho$ we already have to get the Hamiltonian we need. Choose $\alpha_L(u)=\alpha_*(u/L)$, where the fixed smooth cutoff $\alpha_*$ is $1$ on the unit ball and $0$ outside the ball of radius $2$. Then
\begin{gather}
\alpha_L=1\ (|u|\leq L),\qquad\alpha_L=0\ (|u|\geq2L),\nonumber\\
\|D\alpha_L\|\leq C/L,\qquad\|D^2\alpha_L\|\leq C/L^2.\nonumber
\end{gather}
Set
\begin{center}\label{eq:deformation}
$S_{\rho,\lambda}=S_\rho+\lambda\alpha_L(u)\delta g_\rho(V),\qquad
W_{\rho,\lambda}=S_{\rho,\lambda}-u\cdot V,\quad0\leq\lambda\leq1.$
\end{center}
 Moreover $\delta g_\rho\to\delta g_0$ in $C^2$. Fix
\begin{equation}
B=1+\max\{\|\delta g_{0,+}\|_{C^2},\|\delta g_{0,-}\|_{C^2}\}.\nonumber
\end{equation}
The order of choices is
\begin{equation}\label{eq:order}
B\quad\longrightarrow\quad L\quad\longrightarrow\quad\rho.\nonumber
\end{equation}
After the final choice, both corrections have $C^2$ norm at most $B$. Then we first reverse the deformation to the level of generating functions.

\begin{proposition}
Choose $L$ sufficiently large, then $\rho$ sufficiently small. Then, throughout the deformation, the mixed Hessian $S_{\rho,\lambda,uV}$ is invertible. Moreover, there are no critical points of $W_{\rho,\lambda}$ in $L\leq|u|\leq2L$; and at $\lambda=1$ each modified chart has exactly one critical point.
\end{proposition}
\begin{proof}
The mixed Hessian changes by
\begin{center}
$S_{\rho,\lambda,uV}-S_{\rho,uV}
=\lambda(D_u\alpha_L)(D_V\delta g_\rho)^T,$
\end{center}
whose norm is at most $CB/L$. Since $S_{\rho,uV}\to D^{-1}=\operatorname{diag}(-2,2)$, a sufficiently small perturbation remains invertible. Likewise the change in $W_{uu}$ is at most $CB/L^2$, so its definiteness persists.

In the transition region the limiting gradient is
\begin{center}
$W_{0,u}=Hu+(D^{-1}-I)V,\qquad H=\mp2\sqrt3 I.$
\end{center}
Because $V$ is in a fixed compact set, there is a constant $C_0$ independent of $L$ such that
\begin{center}\label{eq:gradientbound}
$|W_{\rho,\lambda,u}|\geq2\sqrt3L-C_0-\frac{CB}{L}-o_\rho(1)>0.$
\end{center}
The term $o_\rho(1)$ is taken after fixing $L$. This excludes all critical points in the transition region. For $|u|\geq2L$ the function is unchanged; the same limiting-gradient estimate excludes critical points there within the parameter chart once $L$ is large and $\rho$ small.

In the core, $\alpha_L\equiv1$, so $D_u\alpha_L=0$ and
\begin{center}\label{eq:core}
$W_{\rho,\lambda,u}=W_{\rho,u}.$
\end{center}
This equality holds componentwise for both $u_1$ and $u_2$: the added term $\lambda\delta g_\rho(V)$ depends only on $V$. The $u$-critical equation therefore still has solution $u=u_\rho(V)$, which lies strictly inside the core. The remaining reduced function is exactly
\begin{center}
$W_{\rho,\lambda}(u_\rho(V),V)=g_\rho(V)+\lambda\delta g_\rho(V).$
\end{center}
At $\lambda=1$, the preceding section proves that this has exactly one critical point in $D_1$.
\end{proof}

Replace $S_\rho$ by $S_{\rho,\lambda}$ in the two mixed charts of the graph of $\Phi_\rho$. Since the correction is supported away from the parameter boundaries, the modified parameterizations agree with the original one on collars and glue smoothly. 

Use the original $M$ as the abstract parameter manifold, and denote the resulting map into $M\times M$ by
\begin{center}
$\iota_\lambda=(p_\lambda,q_\lambda):M\longrightarrow M\times M.$
\end{center}
The first component records the input; the second records the output. In a mixed chart, we use the definition
\begin{center}\label{eq:iota}
$\iota_\lambda(u,V)=\bigl(u,S_{\rho,\lambda,u}(u,V);S_{\rho,\lambda,V}(u,V),V\bigr),$
\end{center}
and we can glue these two definitions in the collar part. We want to show that the graph under correction is also a graph of some symplectic diffeomorphism, so we can reverse it to the diffeomorphism level. On the product, use
\begin{center}
$\Omega_M=-\pi_1^*\omega+\pi_2^*\omega.$
\end{center}
In local coordinates,
\begin{align}
\frac{1}{\rho^2}\iota_\lambda^*\Omega_M
&=-\sum_i\dd u_i\wedge\dd(S_{u_i})+
\sum_i\dd(S_{V_i})\wedge\dd V_i\nonumber\\
&=\dd\left(\sum_iS_{u_i}\,\dd u_i+\sum_iS_{V_i}\,\dd V_i\right)
=\dd(\dd S)=0.\nonumber
\end{align}
The parameterization is an immersion, since $u$ and $V$ can be read back from its coordinates. It has dimension $4$, half the dimension of $M\times M$, and is therefore Lagrangian. Outside the modified charts it is the original Lagrangian graph, so this property holds globally.

\begin{proposition}
For every $\lambda\in[0,1]$, $\iota_\lambda(M)$ is the graph of a global symplectic diffeomorphism $\Phi_{\rho,\lambda}$.
\end{proposition}
\begin{proof}
In a mixed chart,
\begin{center}
$p_\lambda(u,V)=(u,S_{\rho,\lambda,u}(u,V)),\qquad
Dp_\lambda=\begin{pmatrix}I&0\\S_{\rho,\lambda,uu}&S_{\rho,\lambda,uV}\end{pmatrix}.$
\end{center}
The derivative is invertible. Outside the charts the projection is unchanged and locally invertible as well, and they can glue together in the collar part. Thus $p_\lambda:M\to M$ is a local diffeomorphism everywhere.

Since $M$ is compact, this map is proper. Its image is both open and closed in connected $M$, so it is surjective, and a proper local diffeomorphism is a finite covering. Initially $p_0=\id_M$ under the chosen abstract parameterization. Hence $p_\lambda$ is homotopic to the identity and has degree $1$. Its derivative stays nonsingular throughout the isotopy, so its orientation sign remains positive. Each sheet therefore contributes $+1$ to the degree, and the covering has exactly one sheet. Thus $p_\lambda$ is a global diffeomorphism.

Define
\begin{center}
$\Phi_{\rho,\lambda}=q_\lambda\circ p_\lambda^{-1}.$
\end{center}
Every input now has exactly one output, and the parameterized relation is precisely its graph; in particular the immersion has no self-intersections. The Lagrangian identity says
\begin{center}
$p_\lambda^*\omega=q_\lambda^*\omega
=p_\lambda^*(\Phi_{\rho,\lambda}^*\omega).$
\end{center}
Since $p_\lambda$ is a diffeomorphism,
\begin{center}
$\Phi_{\rho,\lambda}^*\omega=\omega.$
\end{center}
This first proves that $\Phi_{\rho,\lambda}$ is a local symplectic diffeomorphism. To state its global invertibility explicitly, compactness again makes it a covering. It is homotopic to the original symplectic diffeomorphism $\Phi_\rho$, so it has degree $1$; symplecticity preserves orientation, and this covering also has one sheet. Thus it is a global symplectic diffeomorphism.
\end{proof}

The path $\Phi_{\rho,\lambda}\Phi_{\rho,0}^{-1}$ is a symplectic isotopy starting at the identity. If $X_\lambda$ is its generating vector field, then $\iota_{X_\lambda}\omega$ is closed. Since
\begin{center}
$H^1(S^2\times S^2;\R)=0,$
\end{center}
it is exact. Its primitives may be normalized at a base point and chosen smoothly in $\lambda$, giving a Hamiltonian isotopy. As $\Phi_{\rho,0}=\Phi_\rho$ is Hamiltonian, the endpoint is Hamiltonian too.

Now we are in the position to complete the proof of Theorem 1.1.

\begin{proof}
    The construction of Section 5 places the four original fixed points in the two core charts. The relation outside the modified parameter regions agrees with the original graph. Fixed points are intersections with the diagonal, and hence correspond locally to critical points of $W_{\rho,1}$. We leave one such point in each chart and none in the transition or outer parts. 
By the preceding propositions, the final count is therefore
\begin{center}
$\#\Fix(\Phi_{\rho,1})=2.$
\end{center}

Concatenate a Hamiltonian path from the identity to $\Phi_\rho$ with the deformation above. Smooth time reparameterizations constant near the joining times and endpoints give a smooth Hamiltonian vanishing near the endpoints of the unit interval. Periodic extension then realizes the endpoint map as the time-one map of a smooth $1$-periodic Hamiltonian.
\end{proof}

\section{A construction in the unequal-area case}

In  this section we will show a similar counterexample when we use the symplectic form with different areas on two factors. This example is much simpler than the equal area one. 

\begin{theorem}
For $0<a<b$, there exists a Hamiltonian diffeomorphism $\Phi$ of $(S^2\times S^2,a\pi_1^*\omega_{S^2}+b\pi_2^*\omega_{S^2}),$ with exactly two fixed points. It is the time-one map of a smooth Hamiltonian $H\in C^\infty((\R/\mathbb Z)\times M)$.
\end{theorem}
For $(j,k)\in S^2\times S^2$, denote $aj+bk$ by $S$ and $ 2arcsin\frac{a}{b}$ by $\theta$. $|S|$ is denoted by $r$. We can replace $R_1\times R_2$ by $R_{e_x,\pi}\times R_{e_y,\theta}$, and replace $\sigma$ by a rotation through angle $\pi$ around $n = \frac{S}{|S|}$ on each factor, denoted by $\sigma'$.

\begin{lemma}
    $\sigma'\circ(R_{e_x,\pi}\times R_{e_y,\theta})$ is a Hamiltonian diffeomorphism.
\end{lemma}
It is a little different from the equal area case, in which $\sigma$ is not Hamiltonian. But if $a\neq b$, the rotation $\sigma'$ is equal to $id$ on the anti-diagonal, where $\sigma$ will reverse it. And $\sigma'$ converges to $\sigma$ away from the anti-diagonal as $\frac{b}{a}\rightarrow 1$. So $\sigma'$ acts trivially on the homology group, while $\sigma$ acts non-trivially and cannot be realized as a Hamiltonian diffeomorphism.
\begin{proof}
The Hamiltonians of $(R_{e_x,\pi}\times R_{e_y,\theta})$ and of $\sigma'$ are
\begin{center}

$A(j,k)=\pi a j_x+\theta b k_y,
\qquad
B(j,k)=-\pi|aj+bk|$

\end{center}
respectively.  In fact, just like the equal area case, we have
\begin{center}\label{eq:XA}
$X_A(j,k)=(\pi e_x\times j,\theta e_y\times k),$\\
$X_B(j,k)=\bigl(-\pi n\times j,-\pi n\times k\bigr).$
\end{center}
Notice that the factors $a$ and $b$ in $A$, $B$ cancel the corresponding weights in the symplectic form.
And the time-one maps are $\bigl(R_{e_x,\pi}j,R_{e_y,\theta}k\bigr),\ \sigma'$.
The composition  $\sigma'\circ(R_{e_x,\pi}\times R_{e_y,\theta})$ is therefore a Hamiltonian diffeomorphism.
\end{proof}

\begin{lemma}
     The fixed-point set of $\sigma'\circ(R_{e_x,\pi}\times R_{e_y,\theta})$ is
\begin{center}
    $\Fix(\sigma'\circ(R_{e_x,\pi}\times R_{e_y,\theta}))=
\left\{
\left(-\varepsilon e_z,\,
      \frac{\varepsilon}{b}(c e_x+a e_z)\right)
:\ \varepsilon\in\{+1,-1\}
\right\}, \qquad c = \sqrt{b^2-a^2}.$
\end{center}
In particular, $\sigma'\circ(R_{e_x,\pi}\times R_{e_y,\theta})$ has exactly two fixed points.
\end{lemma}
\begin{proof}
Use weighted coordinates
\[
J=aj,\qquad K=bk,
\]
which identify $M$ with
\[
\mathcal M_{a,b}
=\{(J,K)\in\mathbb R^3\times\mathbb R^3:
  |J|=a,\ |K|=b\}.
\]
We use the same letters for the maps in these coordinates.
Put
\[
S=J+K,\qquad Q=|S|^2\geq(b-a)^2>0,
\qquad R_x=R_{e_x,\pi},\qquad R_y=R_{e_y,\theta}.
\]
Moreover,
\begin{equation}\label{eq:T}
\sigma'(J,K)=
\left(
-J+\left(1-\frac{c^2}{Q}\right)S,\;
-K+\left(1+\frac{c^2}{Q}\right)S
\right).\nonumber
\end{equation}
Because $(\sigma')^2=\id$, we have the equivalence
\[
\sigma'\circ(\R_x\times R_y)(J,K)=(J,K)
\quad\Longleftrightarrow\quad
(\R_x\times R_y)(J,K)=\sigma'(J,K).
\]
Consequently, the fixed-point equations are exactly
\begin{align}
(R_x+I)J&=\left(1-\frac{c^2}{Q}\right)S,
\\
(R_y+I)K&=\left(1+\frac{c^2}{Q}\right)S.
\end{align}

Let $E_yv=e_y\times v$.
Since $\theta\in(0,\pi)$, the matrix $R_y+I$ is invertible.
Using $\mu=\tan(\theta/2)=\sin\theta/(1+\cos\theta)$, one checks that
\[
(R_y+I)^{-1}
=\frac12
\begin{pmatrix}
1&0&-\mu\\
0&1&0\\
\mu&0&1
\end{pmatrix}
=\frac12(I-\mu E_y).
\]
Equation (7.2) is therefore equivalent to
\begin{equation}
K=\lambda(S-\mu e_y\times S),
\qquad
\lambda=\frac{Q+c^2}{2Q}>0.\nonumber
\end{equation}
Since $R_x$ is a half-turn about $e_x$,
\[
(R_x+I)J=2J_xe_x.
\]
If $Q\neq c^2$, the coefficient on the right-hand side
of (7.1) is nonzero. Thus $S$ must be parallel to $e_x$.
In particular,
\[
S\cdot(e_y\times S)=0,
\qquad
|e_y\times S|^2=|S|^2=Q.
\]
Taking squared norms in $K=\lambda(S-\mu e_y\times S)$ , and using $|K|=b$, gives
\[
b^2=\lambda^2Q(1+\mu^2).
\]
Together with $1+\mu^2
=1+\frac{a^2}{c^2}
=\frac{b^2}{c^2}$, we obtain
\[
(Q+c^2)^2=4Qc^2
\quad\Longleftrightarrow\quad
(Q-c^2)^2=0.
\]
This contradicts $Q\neq c^2$. Hence every fixed point satisfies
$Q = c^2$. When $Q=c^2$, we have $\lambda=1$, so
\[
K=S-\mu e_y\times S,
\qquad
J=S-K=\mu e_y\times S.
\]
Writing $S=(S_x,S_y,S_z)$ and taking squared norms of $J$ yields
\[
a^2
=\mu^2|e_y\times S|^2
=a^2-\mu^2S_y^2.
\]
Since $\mu>0$, it follows that $S_y=0$.
Meanwhile, (7.1) and $Q=c^2$ give $J_x=0$.
As
\[
J=\mu(e_y\times S)
=\mu(S_z,0,-S_x),
\]
we also obtain $S_z=0$.
Finally, $|S|^2=c^2$ implies
\[
S=\varepsilon c e_x,\qquad \varepsilon\in\{+1,-1\}.
\]
Therefore every fixed point must be one of the two pairs
\begin{equation}\label{eq:candidates}
J_\varepsilon=-\varepsilon a e_z,
\qquad
K_\varepsilon=\varepsilon(c e_x+a e_z).\nonumber
\end{equation}
Both pairs lie in $\mathcal M_{a,b}$ because
$|J_\varepsilon|=a$ and
$|K_\varepsilon|^2=c^2+a^2=b^2$.

\end{proof}

 Therefore, unlike the positive result, in this case the non-monotone situation is much simpler than the monotone one.

\newpage

\phantomsection
\addcontentsline{toc}{section}{References}
\bibliographystyle{plain}
\bibliography{ref.bib}
\end{document}